\documentclass[12pt]{amsart}
\usepackage{amscd,amssymb,verbatim}
\usepackage{amssymb,amsmath,amsthm,epsfig}
\newcommand{\N}{\mathbb{N}}

\def\sideremark#1{\ifvmode\leavevmode\fi\vadjust{\vbox
to0pt{\vss \hbox to 0pt{\hskip\hsize\hskip1em
\vbox{\hsize2cm\tiny\raggedright\pretolerance10000
\noindent#1\hfill}\hss}\vbox to8pt{\vfil}\vss}}}

\newcommand{\vp}{\varepsilon}

\newtheorem{thm}{Theorem}[section]
\newtheorem{lem}[thm]{Lemma}
\newtheorem{cor}[thm]{Corollary}
\newtheorem{prop}[thm]{Proposition}
\theoremstyle{definition}

\newtheorem{defn}[thm]{Definition}          

\newcommand{\supp}{{\rm supp}}
\newcommand{\coo}{c_{00}}

\def\sfrac#1#2{\kern.1em\raise.5ex\hbox{$#1$}
        \kern-.1em/\kern-.05em\lower.25ex\hbox{$#2$}}

\title{Upper and lower estimates for Schauder frames and atomic decompositions.}

\date{\today}

\author{K. Beanland}
\address{Department of Mathematics and Applied Mathematics\\ Virginia Commonwealth University, Richmond, VA 23284}
\email{kbeanland@vcu.edu}

\author{D. Freeman}
\address{Department of Mathematics\\ The University of Texas at Austin, Austin, TX 78712-0257}
\email{freeman@math.utexas.edu}

\author{R. Liu}
\address{Department of Mathematics and LPMC\\ Nankai University, Tianjin 300071, P.R.
China} \email{ruiliu@nankai.edu.cn}

\thanks{The research of the first author was supported by a grant from the National Science Foundation.  The third author is supported by
the National Nature Science Foundation of China, the Fundamental
Research Fund for the Central Universities of China, and the Tianjin
Science \& Technology Fund.}

\keywords{Banach spaces; Frames; Atomic Decompositions}

\subjclass[2000]{Primary 46B20, Secondary 41A65}

\begin{document}

\begin{abstract}
We prove that a Schauder frame for any separable Banach space is
shrinking if and only if it has an associated space with a shrinking
basis, and that a Schauder frame for any separable Banach space is
shrinking and boundedly complete if and only if it has a reflexive
associated space.  To obtain these results, we prove that the upper
and lower estimate theorems for finite dimensional decompositions of
Banach spaces can be extended and modified to Schauder frames.  We
show as well that if a separable infinite dimensional Banach space
has a Schauder frame, then it also has a Schauder frame which is not
shrinking.

\end{abstract}\maketitle

\section{Introduction}\label{S:0}
  Frames for Hilbert spaces were
introduced by Duffin and Schaeffer  in 1952 \cite{DS} to address
some questions in non-harmonic Fourier series.  However, the current
popularity of frames is largely due to their successful application
to signal processing, initiated by Daubechies, Grossmann, and Meyer
 in 1986 \cite{DGM}.  A {\em frame} for an infinite dimensional
separable Hilbert space $H$ is a sequence of vectors
$(x_i)_{i=1}^\infty\subset H$ for which there exists constants
$0\leq A\leq B$ such that for any $x\in H$,
\begin{equation}
A\|x\|^2\leq \sum_{i=1}^\infty |\langle x, x_i\rangle|^2\leq
B\|x\|^2.
\end{equation}
If $A=B=1$, then $(x_i)_{i=1}^\infty$ is called a {\em Parseval}
frame. Given any frame $(x_i)_{i=1}^\infty$ for a Hilbert space $H$,
 there exists a frame $(f_i)_{i=1}^\infty$ for $H$, called an {\em
alternate dual frame}, such that for all $x\in H$,
\begin{equation}\label{E:0}
x=\sum_{i=1}^\infty \langle x, f_i\rangle x_i.
\end{equation}
The equality in (\ref{E:0}) allows the reconstruction of any vector
$x$ in the Hilbert space from the sequence of coefficients $(\langle
x, f_i\rangle)_{i=1}^\infty$.  The standard method to construct such
a frame $(f_i)_{i=1}^\infty$ is to take $f_i=S^{-1} x_i$ for all
$i\in\N$, where $S$ is the positive, self-adjoint invertible
operator on $H$ defined by $Sx=\sum_{i=1}^\infty \langle x,
x_i\rangle x_i$ for all $x\in H$. The operator $S$ is called the
{\em frame operator} and the frame $(S^{-1} x_i)_{i=1}^\infty$ is
called the {\em canonical dual frame} of $(x_i)_{i=1}^\infty$.

In their AMS memoir \cite{HL}, Han and Larson initiated studying the
dilation viewpoint of frames.  That is, analyzing frames as
orthogonal projections of Riesz bases, where a Riesz basis is an
unconditional basis for a Hilbert space. To start this approach,
they proved the following theorem.
\begin{thm}[\cite{HL}]\label{T:HilbertDilation}
If $(x_i)_{i=1}^\infty$ is a frame for a Hilbert space $H$, then
there exists a larger Hilbert space $Z\supset H$ and a Riesz basis
$(z_i)_{i=1}^\infty$ for $Z$ such that $P_X z_i=x_i$ for all
$i\in\N$, where $P_X$ is orthogonal projection onto $X$.
Furthermore, if $(x_i)_{i=1}^\infty$ is Parseval, then
$(z_i)_{i=1}^\infty$ can be taken to be an ortho-normal basis.
\end{thm}

The concept of a frame was extended to Banach spaces in 1989 by
Grochenig \cite{G} through the introduction of atomic
decompositions. The main goal of the paper was to obtain for Banach
spaces the unique association of a vector with the natural set of
frame coefficients.  In 2008, Schauder frames for Banach space were
developed \cite{CDOSZ} with the goal of creating a procedure to
represent vectors using quantized coefficients. A Schauder frame
essentially takes, as its definition, an extension of the equation
(\ref{E:0}) to Banach spaces.

\begin{defn}
Let $X$ be an infinite dimensional separable Banach space. A
sequence $(x_i,f_i)_{i=1}^\infty\subset X\times X^*$ is called a
{\em Schauder frame} for $X$ if $x=\sum_{i=1}^\infty f_i(x) x_i$ for
all $x\in X$.
\end{defn}

In particular, if $(x_i)_{i=1}^\infty$ and $(f_i)_{i=1}^\infty$ are
frames for a Hilbert space $H$, then $(f_i)_{i=1}^\infty$ is an
alternate dual frame for $(x_i)_{i=1}^\infty$ if and only if
$(x_i,f_i)_{i=1}^\infty$ is a Schauder frame for $H$.  As noted in
\cite{CDOSZ}, a separable Banach space has a Schauder frame if and
only if it has the bounded approximation property.  By the uniform
boundedness principle, for any Schauder frame
$(x_i,f_i)_{i=1}^\infty$ of a Banach space $X$, there exists a
constant $C\geq 1$ such that $\sup_{n\geq m}\|\sum_{i=m}^n
f_i(x)x_i\|\leq C\|x\|$ for all $x\in X$. The least such value $C$
is called the {\em frame constant} of $(x_i,f_i)_{i=1}^\infty$. A
Schauder frame $(x_i,f_i)_{i=1}^\infty$ is called {\em
unconditional} if the series $x=\sum_{i=1}^\infty f_i(x) x_i$
converges unconditionally for all $x\in X$.
 The following definition
extends the notion of a basis being shrinking or boundedly complete to the context of frames.

\begin{defn}
Given a Schauder frame $(x_i,f_i)_{i=1}^\infty\subset X\times X^*$,
let $T_n:X\rightarrow X$ be the operator $T_n(x)=\sum_{i\geq
n}f_i(x) x_i$.  The frame $(x_i,f_i)_{i=1}^\infty$ is called {\em
shrinking} if $\|x^*\circ T_n\|\rightarrow0$ for all $x^*\in X^*$.
The frame $(x_i,f_i)_{i=1}^\infty$ is called {\em boundedly
complete} if $\sum_{i=1}^\infty x^{**}(f_i)x_i$ converges in norm to
an element of $X$ for all $x^{**}\in X^{**}$.
\end{defn}
As noted in \cite{CL}, if $(x_i)_{i=1}^\infty$ is a Schauder basis
and $(x^*_i)_{i=1}^\infty$ are the biorthogonal functionals of
$(x_i)_{i=1}^\infty$, then the frame $(x_i,x^*_i)_{i=1}^\infty$ is
shrinking if and only if the basis $(x_i)_{i=1}^\infty$ is
shrinking, and the frame $(x_i,x^*_i)_{i=1}^\infty$ is boundedly
complete if and only if the basis $(x_i)_{i=1}^\infty$ is boundedly
complete. Thus the definition of a frame being shrinking or
boundedly complete is consistent with that of a basis. In \cite{L},
the frame properties shrinking and boundedly complete are called
pre-shrinking and pre-boundedly complete.  The following definitions
allow the dilation viewpoint of Han and Larson to be extended to
Schauder frames.

\begin{defn}\label{D:0}
Let $(x_i,f_i)_{i=1}^\infty$ be a frame for a Banach space $X$ and
let $Z$ be a Banach space with basis $(z_i)_{i=1}^\infty$ and
coordinate functionals $(z_i^*)_{i=1}^\infty$. We call $Z$ an {\em
associated space} to $(x_i,f_i)_{i=1}^\infty$ and
$(z_i)_{i=1}^\infty$ an {\em associated basis }if the operators
$T:X\rightarrow Z$ and $S:Z\rightarrow X$ are bounded, where,
$T(x)=T(\sum f_i(x)x_i)=\sum f_i(x)z_i$ for all $x\in X$ and
$S(z)=S(\sum z^*_i(z) z_i)=\sum z^*_i(z) x_i$ for all $z\in Z.$
\end{defn}

Essentially, Theorem \ref{T:HilbertDilation} states that a frame for
a Hilbert space has an associated basis which is a Riesz basis for a
Hilbert space. Furthermore, the proof in \cite{HL} actually involves
constructing the operators $T$ and $S$ given in Definition
\ref{D:0}.  In \cite{CDOSZ}, it is shown that every Schauder frame
has an associated space, which is referred to as the minimal
associated space in \cite{L}. Furthermore, the minimal associated
basis will be unconditional if and only if the Schauder frame is
unconditional. On the other hand, if $(x_i,f_i)$ is a Schauder
frame, then the reconstruction operator, $S$, for the minimal
associated space contains $c_0$ in its kernel if and only if a
finite number of vectors can be removed from $(x_i)$ to make it a
Schauder basis \cite{LZ}.  Thus, except in trivial cases, the
minimal associated basis will not be boundedly complete. Given some
desirable property for a basis to have, it is natural to consider
the problem of characterizing whether or not a particular Schauder
frame has an associated basis with that property. It is not
difficult to see that if a Schauder frame has a shrinking associated
basis, then the frame must be shrinking as well, and that
 if a Schauder frame has a boundedly complete associated basis, then the frame must be boundedly
 complete.  Under additional assumptions, it is proven in \cite{L} that the minimal associated space to a frame is shrinking if the frame itself is shrinking
and that the maximal associated space to a frame is boundedly complete if the frame itself is boundedly complete.
One of our main theorems is to give the following complete characterization.

\begin{thm}\label{T:main}
Let $(x_i,f_i)_{i=1}^\infty$ be a Schauder frame for a Banach space
$X$.  Then $(x_i,f_i)_{i=1}^\infty$ is shrinking if and only if
$(x_i,f_i)_{i=1}^\infty$ has a shrinking associated basis.
Furthermore, $(x_i,f_i)_{i=1}^\infty$ is shrinking and boundedly
complete if and only if $(x_i,f_i)_{i=1}^\infty$ has a reflexive
associated space.
\end{thm}

To obtain Theorem \ref{T:main}, we prove a stronger result, involving quantitative bounds on the Szlenk index.
  The Szlenk index \cite{Sz} is a co-analytic rank on
the set of Banach spaces with separable dual, and was created to
prove that there does not exist a single Banach space $X$ with
separable dual such that every Banach space with separable dual is
isomorphic to a subspace of $X$.  In particular, the Szlenk index of
a Banach space is countable if and only if the Banach space has
separable dual. In \cite{OSZ2}, it is shown that the higher order
Tsirelson spaces $T_{\alpha,c}$, where $\alpha$ is a countable
ordinal and $0<c<1$, can be used to measure the Szlenk index through
the use of tree estimates.  Furthermore, a Banach space with
separable dual has Szlenk index at most $\omega^{\alpha\omega}$ for
some given countable ordinal $\alpha$, if and only if the Banach
space satisfies subsequential $T_{\alpha,c}$-upper tree estimates
for some constant $0 < c < 1$. Thus proving theorems about upper
tree estimates provides quantitative insight into Banach spaces with
separable dual, and similarly, proving theorems about upper and
lower tree estimates provides quantitative insight into separable
reflexive Banach spaces.  In \cite{OSZ1}, a characterization is
given for when a separable reflexive Banach space embeds into a
Banach space with an FDD satisfying certain upper and lower block
estimates, and in \cite{FOSZ}, a characterization is given for when
a Banach space with separable dual embeds into a Banach space with
an FDD satisfying certain upper block estimates.  We extend both of
those theorems to frames and by applying them to $T_\alpha$ and
$T_\alpha^*$,
 we obtain the following two charachterizations.

\begin{thm}\label{T:mainUpper}
Let $(x_i,f_i)_{i=1}^\infty$ be a shrinking Schauder frame for a Banach space $X$ and let $\alpha$ be a countable ordinal.
Then, the following are equivalent.
\begin{enumerate}
\item [(a)] $X$ has Szlenk index at most $\omega^{\alpha\omega}$.
\item [(b)] $X$ satisfies
subsequential $T_{\alpha,c}$-upper tree estimates for some constant
$0<c<1$.
\item [(c)] $(x_i,f_i)_{i=1}^\infty$ has an associated basis $(z_i)_{i=1}^\infty$ such that there exists $(n_i)_{i=1}^\infty,(K_i)_{i=1}^\infty\in[\N]^\omega$ and $0<c<1$
so that the FDD $(span_{j\in[n_i,n_{i+1})}z_i)_{i=1}^\infty$
satisfies subsequential
 $(t_{K_i})_{i=1}^\infty$ upper block estimates, where $(t_i)_{i=1}^\infty$ is the unit vector basis for $T_{\alpha,c}$.
\end{enumerate}
\end{thm}

\begin{thm}\label{T:mainReflexive}

Let $(x_i,f_i)_{i=1}^\infty$ be a shrinking and boundedly complete Schauder frame for a Banach space $X$.  Then, the following are equivalent.
\begin{enumerate}
\item[(a)] $X$ and $X^*$ both have Szlenk index at most $\omega^{\alpha\omega}$.
\item[(b)] $X$ satisfies
subsequential $T_{\alpha,c}$-upper tree estimates and subsequential
$T^*_{\alpha,c}$-lower tree estimates  for some constant $0<c<1$.
\item[(c)] $(x_i,f_i)_{i=1}^\infty$ has an associated basis $(z_i)_{i=1}^\infty$ such that there exists
$(n_i)_{i=1}^\infty,(K_i)_{i=1}^\infty\in[\N]^\omega$ and $0<c<1$ so
that the FDD $(span_{j\in[n_i,n_{i+1})}z_i)_{i=1}^\infty$ satisfies
subsequential
 $(t_{K_i})_{i=1}^\infty$ upper block estimates and subsequential
 $(t^*_{K_i})_{i=1}^\infty$ lower block estimates, where $(t_i)_{i=1}^\infty$ is the unit vector basis for $T_{\alpha,c}$.
\end{enumerate}
\end{thm}

The Banach space $T_{\alpha,c}$ is reflexive, and thus the basis
$(z_i)_{i=1}^\infty$ given in Theorem \ref{T:mainUpper} is shrinking
and the basis $(z_i)_{i=1}^\infty$ given in Theorem
\ref{T:mainReflexive} is shrinking and boundedly complete.  Thus
Theorem \ref{T:main} follows immediately from Theorems
\ref{T:mainUpper} and \ref{T:mainReflexive}.

Both Schauder frames and atomic decompositions are natural
extensions of frame theory into the study of Banach space. These two
concepts are directly related, and some papers in the area such as
\cite{CHL}, \cite{CL}, and \cite{CLS} are stated in terms of atomic
decompositions, while others such as \cite{CDOSZ}, \cite{L}, and
\cite{LZ} are stated in terms of Schauder frames.

\begin{defn}\label{D:AD}
Let $X$ be a Banach space and $Z$ be a Banach sequence space. We say
that a sequence of pairs $(x_i,f_i)_{i=1}^\infty\subset X\times X^*$
is an {\em atomic decomposition} of $X$ with respect to $Z$ if there
exists positive constants $A$ and $B$ such that for all $x \in X$:
\begin{enumerate}
\item [(a)] $(f_i(x_i))_{i=1}^\infty\in Z$,
\item [(b)] $A\|x\|\leq \|(f_i(x_i))_{i=1}^\infty\|_Z\leq B\|x\|$,
\item [(c)] $x=\sum_{i=1}^\infty f_i(x) x_i$.
\end{enumerate}
\end{defn}

If the unit vectors in the Banach space $Z$ given in Definition
\ref{D:AD} form a basis for $Z$, then an atomic decomposition is
simply a Schauder frame with a specified associated space $Z$.
    We choose to use
the terminology of Schauder frames for this paper instead of atomic
decomposition as to us, an associated space is an object which is
external to the space $X$ and frame $(x_i,f_i)_{i=1}^\infty$.  Our
goals are essentially, to construct `nice' associated spaces, given
a particular Schauder frame.  However, our theorems can be stated in
terms of atomic decompositions.  In particular, Theorem \ref{T:main}
can be stated as the following.

\begin{thm}
 Let $X$ be a Banach space and $Z$ be a Banach sequence space whose unit vectors form a basis for $Z$.
Let $(x_i,f_i)_{i=1}^\infty$ be an atomic decomposition of $X$ with
respect to $Z$.  Then $(x_i,f_i)_{i=1}^\infty$ is shrinking if and
only if there exists a Banach sequence space $Z'$ whose unit vectors
form a shrinking basis for $Z'$ such that $(x_i,f_i)_{i=1}^\infty$
is an atomic decomposition of $X$ with respect to $Z'$. Furthermore,
$(x_i,f_i)_{i=1}^\infty$ is shrinking and boundedly complete if and
only if there exists a reflexive Banach sequence space $Z'$ whose
unit vectors form a basis for $Z'$ such that
$(x_i,f_i)_{i=1}^\infty$ is an atomic decomposition of $X$ with
respect to $Z'$.
\end{thm}

The third author would like to express his gratitude to Edward Odell
for inviting him to visit the University of Texas at Austin in the
fall of 2010 and spring of 2011.

\section{Shrinking and boundedly complete Schauder Frames}

It is well known that a basis $(x_i)$ for a Banach space $X$ is shrinking if and only if the biorthogonal functionals $(x_i^*)$ form a
boundedly complete basis for $X^*$.
The following theorem extends this useful characterization to Schauder frames.

\begin{thm}\cite[Proposition 2.3]{CL}\cite[Proposition 4.8]{L}\label{T:CL2.3}
Let $X$ be a Banach space with a Schauder frame $(x_i,f_i)_{i=1}^\infty\subset X\times X^*$.  The frame
 $(x_i,f_i)_{i=1}^\infty$ is shrinking
 if and only if $(f_i,x_i)_{i=1}^\infty$ is a boundedly complete Schauder frame for $X^*$.
\end{thm}

It is a classic and fundamental result of James that a basis for a Banach space is both shrinking and boundedly complete,
if and only if the Banach space is reflexive.  The following theorem shows that one side of James' characterization holds for frames.
\begin{thm}\cite[Proposition 2.4]{CL}\cite[Proposition 4.9]{L}\label{T:CL2.4}
If $(x_i,f_i)_{i=1}^\infty$ is a shrinking and boundedly complete Schauder frame of a Banach space $X$, then $X$ is reflexive.
\end{thm}
It was left as an open question in \cite{CL} whether the converse of Theorem \ref{T:CL2.4} holds.  The following theorem shows that
this is false for any Banach space $X$, and is evidence of how general Schauder frames can exhibit fairly unintuitive structure.
\begin{thm}\label{T:0}
Let $X$ be a Banach space which admits a Schauder frame (i.e. has the bounded approximation property), then $X$ has
a Schauder frame which is not shrinking.
\end{thm}
\begin{proof}
Let $(x_i,f_i)_{i=1}^\infty$ be a Schauder frame for $X$.  If
$(x_i,f_i)_{i=1}^\infty$ is not shrinking, then we are done.  Thus
we assume that $(x_i,f_i)_{i=1}^\infty$ is shrinking.  Fix $x\in X$
such that $x\neq0$, and choose $(y^*_i)_{i=1}^\infty\subset S_{X^*}$
such that $y^*_i\rightarrow_{w^*}0$. For all $n\in\N$, we define
elements
$(x_{3n-2}',f_{3n-2}'),(x_{3n-1}',f_{3n-1}'),(x_{3n}',f_{3n}')\in
X\times X^*$, in the following way:
\[\begin{array}{ccc}
\bar{x}_{3n-2}=x_n & \bar{x}_{3n-1}=x & \bar{x}_{3n}=x\\
\bar{f}_{3n-2}=f_n & \bar{f}_{3n-1}=-y^*_n & \bar{f}_{3n}=y^*_n
\end{array}\]
As $y^*_i\rightarrow_{w^*}0$, it is not difficult to see that $(\bar{x}_i,\bar{f}_i)_{i=1}^\infty$ is a frame of $X$.
  However, $(\bar{x}_i,\bar{f}_i)_{i=1}^\infty$
is not shrinking.  Indeed, let $x^*\in X^*$ such that $x^*(x)=1$.  As $(x_i,f_i)_{i=1}^\infty$ is shrinking, there exists $N_0\in\N$ such that
$|\sum_{i=M}^\infty x^*(x_i)f_i(y)|\leq\frac{1}{4}\|y\|$ for all $y\in X$ and $M\geq N_0$.  Let $M\geq N_0$ and choose $y\in B_X$ such that
 $y^*_M(y)>\frac{3}{4}$.  We now have the following estimate,
 $$\|x^*\circ T_{3M}\|\geq x^*\circ T_{3M}(y)= x^*(x)y^*_M(y)+\sum_{i=M+1}^\infty f_i(y)x^*(x_i)>\frac{3}{4}-\frac{1}{4}=\frac{1}{2}
 $$
Thus we have that $\|x^*\circ T_{N}\|\not\rightarrow0$, and hence $(\bar{x}_i,\bar{f}_i)_{i=1}^\infty$ is not shrinking.

\end{proof}

As a Schauder frame must be shrinking in order to have a shrinking
associated basis, Theorem \ref{T:0}
 implies that not every Schauder frame for a reflexive Banach space has a reflexive associated space.

\begin{defn}
If $(x_i,f_i)_{i=1}^\infty$ is a Schauder frame for a Banach space
$X$, and $(z_i)_{i=1}^\infty$ is an associated basis then
$(x_i,f_i)_{i=1}^\infty$ is called strongly shrinking relative to
$(z_i)_{i=1}^\infty$ if
$$\|x^*\circ S_n\|\rightarrow0\qquad\textrm{for all }x^*\in X^*
$$
where $S_n:Z\rightarrow X$ is defined by $S_n(z)=\sum_{i=n}^\infty
z^*_i(z)x_i$.
\end{defn}

It is clear that if a Schauder frame is strongly shrinking relative
to some associated basis, then the Schauder frame must be shrinking.
Also, if a Schauder frame has a shrinking associated basis, then the
frame is strongly shrinking relative to the basis. In \cite{CL},
examples of shrinking Schauder frames are given which are not
strongly shrinking relative to some given associated spaces.
However, we will prove later that for any given shrinking Schauder
frame, there exists an associated basis such that the frame is
strongly shrinking relative to the basis.  Before proving this, we
state the following theorem which illustrates why the concept of
strongly shrinking will be important to us and allows us to use
frames in duality arguments.

\begin{thm}\cite[Lemma 1.7, Theorem 1.8]{CL}\label{T:CL1.8}
If $(x_i,f_i)_{i=1}^\infty$ is a Schauder frame for a Banach space
$X$, and $(z_i)_{i=1}^\infty$ is an associated basis for
$(x_i,f_i)_{i=1}^\infty$ then $(z^*_i)_{i=1}^\infty$ is an
associated basis for $(f_i,x_i)_{i=1}^\infty$, if and only if
$(x_i,f_i)_{i=1}^\infty$ is strongly shrinking relative to
$(z_i)_{i=1}^\infty$.

Furthermore, given operators $T:X\rightarrow Z$ and $S:Z\rightarrow X$ defined by
$T(x)=\sum f_i(x)z_i$ for all $x\in X$ and $S(z)=\sum z^*_i(z) x_i$ for all $z\in Z$, if $(x_i,f_i)_{i=1}^\infty$ is strongly shrinking relative to
$(z_i)_{i=1}^\infty$, then $S^*:X^*\rightarrow [z_i^*]$ and $T^*:[z_i^*]\rightarrow X^*$ are given by
$S^*(x^*)=\sum x^*(x_i)z^*_i$ for all $x^*\in X^*$ and $T^*(z^*)=\sum z^*(z_i)f_i$ for all $z^*\in [z_i^*]$.
\end{thm}

Applying Theorem \ref{T:CL1.8} to reflexive Banach spaces gives the following corollary.

\begin{cor}\label{C:1.1}
If $(x_i,f_i)_{i=1}^\infty$ is a shrinking frame for a reflexive Banach space $X$ and
$(z_i)_{i=1}^\infty$ is an associated basis such that $(x_i,f_i)_{i=1}^\infty$ is strongly shrinking relative to $(z_i)_{i=1}^\infty$ then
$(f_i,x_i)_{i=1}^\infty$ is strongly shrinking relative to $(z_i^*)_{i=1}^\infty$.
\end{cor}

Before proceeding further, we need some stability lemmas. Note that
if $(z_i)_{i=1}^\infty$ is a basis for a Banach space $Z$, with
projection operators $P_{(n,k)}:Z\rightarrow Z$ given by
$P_{(n,k)}(\sum a_i z_i)=\sum_{i\in(n,k)}a_i z_i$, then
$P_{(1,k)}\circ P_{(n,\infty)}=0$ and $P_{(n,\infty)}\circ
P_{(1,k)}=0$ for all $k<n$. The analogous property fails when
working with frames.  However, the following lemmas will essentially
allow us to obtain this property within some given $\vp>0$ if $n$ is
chosen sufficiently larger than $k$.

\begin{lem} \label{L:0}
Let $(x_i,f_i)_{i=1}^\infty$ be a Schauder frame for a Banach space $X$.  Then
 for all $\vp>0$ and $k\in\N$, there exists $N\in\N$ such that $N>k$
 and
 $$\sup_{n\geq m\geq N>k\geq n_0\geq m_0}\|\sum_{i=m}^n f_i(\sum_{j=m_0}^{n_0}f_j(x)x_j)x_i\|\leq\vp\|x\|\qquad\textrm{ for all }x\in X.$$
\end{lem}

\begin{proof}
As $(x_i,f_i)_{i=1}^\infty$ is a Schauder frame, for each $1\leq \ell\leq k$ with $f_\ell\neq0$, there exists
 $N_\ell>k$ such
 that $\sup_{n\geq m\geq N_\ell}\|\sum_{i=m}^n f_i(x_\ell)x_i\|<\vp/(k\|f_\ell\|)$.  Let $N=\max_{1\leq\ell\leq k}N_\ell$.
 We now obtain the following estimate for $n\geq m\geq N>k\geq n_0\geq m_0$ and $x\in X$.
\begin{align*}
\|\sum_{i=m}^n \sum_{j=m_0}^{n_0}f_j(x)f_i(x_j)x_i\|&\leq k\max_{1\leq \ell\leq k}\|\sum_{i=m}^n f_\ell(x)f_i(x_\ell)x_i\|\qquad\quad\textrm{as }k\geq n_0\\
&\leq k\max_{1\leq \ell \leq k}  \|\sum_{i=m}^n f_i(x_\ell) x_i\|\|f_\ell\|\|x\|\leq\vp\|x\|\qquad\textrm{as }n\geq m\geq N_\ell.\\
\end{align*}
\end{proof}

In terms of operators, Lemma \ref{L:0} can be stated as for all
$k\in \N$, if $(x_i,f_i)_{i=1}^\infty$ is a Schauder frame then
$\lim_{N\rightarrow\infty}\|T_N\circ (Id_X-T_k)\|=0$, where
$T_n:X\rightarrow X$ is given by $T_n(x)=\sum_{i=n}^\infty
f_i(x)x_i$ for all $n\in\N$.
  We now prove that for all $k\in \N$, if $(x_i,f_i)_{i=1}^\infty$ is a shrinking Schauder frame then
$\lim_{N\rightarrow\infty}\| (Id_X-T_k)\circ T_N\|= 0$. The frame
given in the proof of Theorem \ref{T:0} shows that we cannot drop
the condition of shrinking.

\begin{lem}\label{L:1}
Let $(x_i,f_i)_{i=1}^\infty$ be a shrinking Schauder frame for a Banach space $X$.  Then
 for all $\vp>0$ and $k\in\N$, there exists $N\in\N$ such that $N>k$
 and
 $$\sup_{n\geq m\geq N>k\geq n_0\geq m_0}\|\sum_{i=m_0}^{n_0} f_i(\sum_{j=m}^n f_j(x) x_j)x_i\|\leq\vp\|x\|\qquad\textrm{ for all }x\in X.$$
\end{lem}
\begin{proof}

By Theorem \ref{T:CL2.3},
 $(f_i,x_i)_{i=1}^\infty$ is a Schauder frame for $X^*$.  Thus for each $1\leq \ell\leq k$ with $x_\ell\neq0$, there exists
 $N_\ell>k$ such
 that $\sup_{n\geq m\geq N_\ell}\|\sum_{j=m}^n f_\ell(x_j)f_j\|<\vp/(k\|x_\ell\|)$.  Let $N=\max_{1\leq\ell\leq k}N_\ell$.
 We now obtain the following estimate for $n\geq m\geq N>k\geq n_0\geq m_0$ and $x\in X$.
\begin{align*}
 \|\sum_{i=m_0}^{n_0} f_i(\sum_{j=m}^n f_j(x) x_j)x_i\|&\leq k\sup_{1\leq \ell \leq k} \| f_\ell(\sum_{j=m}^n f_j(x) x_j)x_\ell\|\qquad\textrm{ as }k\geq n_0\\
&= k\sup_{1\leq \ell \leq k}  \left|\sum_{j=m}^n f_j(x) f_\ell(x_j)\right|\|x_\ell\|\\
&\leq k\sup_{1\leq \ell \leq k}  \|\sum_{j=m}^n f_\ell(x_j) f_j\| \|x\|\|x_\ell\|\leq\vp\|x\|\qquad\textrm{as }n\geq m\geq N_\ell.\\
\end{align*}

\end{proof}

Our method for proving that every shrinking frame has a shrinking associated basis is to first prove that every
shrinking frame is strongly shrinking with respect to some associated basis, and then renorm that associated
basis to make it shrinking.  The following theorem is thus our first major step.

\begin{thm}\label{T:1.2}
Let $(x_i,f_i)_{i=1}^\infty$ be a shrinking Schauder frame for a Banach space $X$.  Then $(x_i,f_i)_{i=1}^\infty$ has an
associated basis $(z_i)_{i=1}^\infty$ such that $(x_i,f_i)_{i=1}^\infty$ is strongly shrinking relative to $(z_i)_{i=1}^\infty$.
\end{thm}
\begin{proof}
We repeatedly apply Lemma \ref{L:1} to obtain a subsequence $(N_k)_{k=1}^\infty$ of $\N$ such that for all $k\in\N$,
\begin{equation}\label{E:2.8.1}
\sup_{n\geq m\geq N_k}\|\sum_{i=1}^k f_i(\sum_{j=m}^n f_j(x) x_j)x_i\|\leq  2^{-2k}\|x\|\qquad\textrm{ for all }x\in X.
\end{equation}

We assume without loss of generality that $x_i\neq0$ for all $i\in\N$.
We denote the unit vector basis of $\coo$ by $(z_i)_{i=1}^\infty$, and
 define the following norm,
$\|\cdot\|_Z$ for all $(a_i)\in\coo$.
\begin{equation}\label{E:2.8.2}
\|\sum a_i z_i\|_Z=\max_{n\geq m}\|\sum_{i=m}^n a_i x_i\|\vee
\max_{k\in\N;n\geq m\geq N_k}2^k \|\sum_{i=1}^{k}f_i(\sum_{j=m}^n
a_j x_j)x_i\|.
\end{equation}
It follows easily that $(z_i)_{i=1}^\infty$ is a bimonotone basic sequence, and thus $(z_i)_{i=1}^\infty$ is a bimonotone
basis for the completion of $\coo$ under $\|\cdot\|_Z$, which we denote by $Z$.
We first prove that $(z_i)_{i=1}^\infty$ is an associated basis for $(x_i,f_i)_{i=1}^\infty$.
Let $C$ be the frame constant of $(x_i,f_i)_{i=1}^\infty$.  That is, $max_{n\geq m}\|\sum_{i=m}^n f_i(x)x_i\|\leq C\|x\|$ for all $x\in X$.
By (\ref{E:2.8.1}) and (\ref{E:2.8.2}), the operator $T:X\rightarrow Z$,
defined by $T(x)=\sum f_i(x)z_i$ for all $x\in X$, is bounded and $\|T\|\leq C$.
We have that $\|\sum_{i=m}^n a_i z_i\|_Z\geq \|\sum_{i=m}^n a_i x_i\|$, and hence the operator $S:Z\rightarrow X$ defined by
 $S(z)=\sum z^*_i(z)x_i$ is bounded and $\|S\|=1$.  Thus we have that $(z_i)_{i=1}^\infty$ is an associated basis for $(x_i,f_i)_{i=1}^\infty$.

 We now prove that $(x_i,f_i)_{i=1}^\infty$ is
strongly shrinking relative to $(z_i)_{i=1}^\infty$.  Let $\vp>0$ and $x^*\in B_{X^*}$.  As $(x_i,f_i)_{i=1}^\infty$ is shrinking, we may
choose $k\in\N$ such that $2^{-k}<\vp/2$ and $\|\sum_{j=k+1}^{\infty}x^*(x_j)f_j\|<\vp/2$. We obtain the following estimate
for any $N\geq N_k$ and $z=\sum a_i z_i\in Z$.
\begin{align*}
x^*(\sum_{i=N}^\infty a_i x_i)&=\sum_{j=1}^{\infty}x^*(x_j)f_j(\sum_{i=N}^\infty a_i x_i)\qquad\textrm{as }(f_i,x_i)_{i=1}^\infty\textrm{ is a frame for }X^*\\
&=\sum_{j=1}^{k}x^*(x_j)f_j(\sum_{i=N}^\infty a_i x_i)+\sum_{j=k+1}^{\infty}x^*(x_j)f_j(\sum_{i=N}^\infty a_i x_i)\\
&\leq \|x^*\|\|\sum_{j=1}^{k}f_j(\sum_{i=N}^\infty a_i x_i)x_j\|+\|\sum_{j=k+1}^{\infty}x^*(x_j)f_j\|\|z\|_Z\qquad\textrm{  as }\|\sum_{i=N}^\infty a_i x_i\|\leq\|z\|_Z\\
&\leq \|x^*\|2^{-k}\|z\|_Z+\|\sum_{j=k+1}^{\infty}x^*(x_j)f_j\|\|z\|_Z\qquad\textrm{ by (\ref{E:2.8.2}) as }N\geq N_k\\
&<\vp/2\|z\|_Z+\vp/2\|z\|_Z
\end{align*}
We thus have that for all $x^*\in X^*$ and $\vp>0$, that there
exists $M\in\N$ such that $|x^*(\sum_{i=N}^\infty z^*_i(z)x_i)|<\vp$
for all $N\geq M$ and $z^*\in B_{Z^*}$.  Hence,
 $(x_i,f_i)_{i=1}^\infty$ is
strongly shrinking relative to $(z_i)_{i=1}^\infty$.

\end{proof}

The following lemmas incorporate an associated basis into the tail and initial segment estimates of Lemma \ref{L:0} and Lemma \ref{L:1}.

\begin{lem}\label{L:2.1}
Let $X$ be a Banach space with a shrinking Schauder frame $(x_i,f_i)_{i=1}^\infty$.
  Let $Z$ be a Banach space with basis $(z_i)_{i=1}^\infty$ such that $(x_i,f_i)_{i=1}^\infty$ is strongly shrinking relative to $(z_i)_{i=1}^\infty$. Then for all $k\in\N$ and $\vp>0$, there exists $N\in\N$ such that
$$\sup_{k\geq n\geq m}\|\sum_{i=N}^{\infty}(\sum_{j=m}^{n}x^*(x_j)f_j(x_i))z^*_i\|<\vp\|x^*\|\qquad\textrm{for all }x^*\in X^*$$
\end{lem}

\begin{proof}
Let $k\in\N$ and $\vp>0$.  By renorming $Z$, we may assume without
loss of generality that $(z_i)_{i=1}^\infty$ is bimonotone.  Let
$K\geq1$ be the frame constant of the frame $(f_i,x_i)_{i=1}^\infty$
for $X^*$. We choose a finite $\frac{\vp}{2\|S\|}$-net
$(y^*_\alpha)_{\alpha\in A}$ in $\{y^*\in K\cdot B_{Y^*}\,:\, y^*\in
span_{1\leq i\leq k}(f_i)\}$.  By Theorem \ref{T:CL1.8}, the bounded
operator $S^*:X^*\rightarrow [z_i^*]$ is given by
$S^*(x^*)=\sum_{i=1}^{\infty}x^*(x_i)z^*_i$ for all $x^*\in X^*$. As
$(z_i^*)_{i=1}^\infty$ is a basis for $[z^*]_{i=1}^\infty$, for each
$\alpha\in A$, there exists $N_\alpha\in\N$ such that
$\|P_{[N_\alpha,\infty)}\circ
S^*(y^*_\alpha)\|=\|\sum_{i=N_\alpha}^{\infty}y_\alpha^*(x_i)z^*_i\|<\frac{\vp}{2}$.
We set $N=\max_{\alpha\in A}N_\alpha$. Given, $x^*\in B_X^*$ and
$m,n\in\N$ such that $k\geq n\geq m$, we choose $\alpha\in A$ such
that $\|y^*_\alpha-\sum_{j=m}^{n}x^*(x_j)f_j\|<\frac{\vp}{2\|S\|}$,
which yields the following estimates.
\begin{align*}
\|\sum_{i=N}^{\infty}(\sum_{j=m}^{n}x^*(x_j)f_j(x_i))z^*_i\|&=\|P_{[N,\infty)}\circ S^*(\sum_{j=m}^{n}x^*(x_j)f_j)\|\\
&\leq\|P_{[N,\infty)}\circ S^*(y^*_\alpha)\|+\|P_{[N,\infty)}\circ S^*(y^*_\alpha-\sum_{j=m}^{n}x^*(x_j)f_j)\|\\
&\leq\|P_{[N,\infty)}\circ S^*(y^*_\alpha)\|+\|P_{[N,\infty)}\|\|S\|\|y^*_\alpha-\sum_{j=m}^{n}x^*(x_j)f_j\|\\
&<\frac{\vp}{2} + \|S\|\frac{\vp}{2\|S\|}=\vp\\
\end{align*}

\end{proof}

\begin{lem}\label{L:2.2}
Let $X$ be a Banach space with a shrinking Schauder frame
$(x_i,f_i)_{i=1}^\infty$ and
  let $Z$ be a Banach space with a basis $(z_i)_{i=1}^\infty$. Then for all $k\in\N$ and
$\vp>0$, there exists $N\in\N$ such that
$$\sup_{n\geq m\geq N}\|\sum_{i=1}^{k}(\sum_{j=m}^{n}x^*(x_j)f_j(x_i))z^*_i\|<\vp\|x^*\|\qquad\textrm{for all }x^*\in X^*$$

\end{lem}
\begin{proof}

Let $k\in\N$ and $\vp>0$. As $(x_j,f_j)_{j=1}^\infty$ is a Schauder frame for $X$, the series $\sum_{j=1}^\infty f_j(x_i)x_j$ converges
in norm to $x_i$ for all $i\in\N$.  Thus there exists $N\in\N$ such that
$\sup_{n\geq m\geq N}\|\sum_{j=m}^n f_j(x_i)x_j\|<\frac{\vp}{k\|z^*_i\|}$ for all $1\leq i\leq k$.  For $x^*\in B_{X^*}$ and $n\geq m\geq N$, we have that
$$
\|\sum_{i=1}^{k}(\sum_{j=m}^{n}x^*(x_j)f_j(x_i))z^*_i\|\leq\sum_{i=1}^{k}\|\sum_{j=m}^{n}f_j(x_i)x_j\|\|z^*_i\|<\sum_{i=1}^k\frac{\vp}{k\|z^*_i\|}\|z_i^*\|=\vp\\
$$
\end{proof}

The following lemma and theorem are based on an idea of W. B.
Johnson \cite{J}, and are analogous to Proposition 3.1 in
\cite{FOSZ}, and Lemma 4.3 in \cite{OS}.  Their importance comes
from allowing us to use arguments that require `skipping
coordinates', and in particular, will allow us to apply Proposition
\ref{P:skipping}.

\begin{lem}\label{L:1.3}
Let $X$ be a Banach space with a boundedly complete Schauder frame
$(x_i,f_i)_{i=1}^\infty\subset X\times X^*$. Let
 $\vp_i\searrow0$ and $(p_i)_{i=1}^\infty\in[\N]^\omega$. There exists $(k_i)_{i=1}^\infty\in[\N]^\omega$
 such that for all $x^{**}\in X^{**}$ and for all $N\in\N$ there exists $M\in\N$ such that $k_N< M<k_{N+1}$ and
 $$\sup_{p_{M+1}> n\geq m\geq {p_{M-1}}} \|\sum_{i=m}^n x^{**}(f_i) x_i\|<\vp_N\|x^{**}\|.
 $$
\end{lem}
\begin{proof}
Assume not, then there exists $\vp>0$ and $K_0\in\N$ such that for
all $K>K_0$ there exists $x^{**}_{K}\in B_{X^{**}}$ such that for
all $K_0< M<K$ there exists $n_{K,M},m_{K,M}\in\N$ with $p_{M-1}\leq
m_{K,M}\leq n_{K,M}< p_{M+1}$ and
 $\|\sum_{i=m_{K,M}}^{n_{K,M}} x_{K}^{**}(f_i) x_i\|>\vp$. As $[p_{M-1},p_{M+1}]$ is finite, we may choose a sequence $(K_i)_{i=1}^\infty\in[\N]^\omega$ such
 that for every $M\in\N$ there exists $n_M,m_M\in\N$
 such that $n_{K_i,M}=n_M$ and $m_{K_i,M}=m_M$ for all $i\geq M$.  After passing to a further subsequence
  of $(K_i)_{i=1}^\infty$, we may assume that there exists $x^{**}\in X^{**}$ such that
  $x^{**}_{K_i}(f_j)\rightarrow x^{**}(f_j)$ for all $j\in\N$.  Thus $\|\sum_{i=m_M}^{n_M} x^{**}(f_i) x_i\|\geq\vp$.
  This contradicts that the series $\sum_{i=1}^{\infty} x^{**}(f_i) x_i$ is norm convergent.
\end{proof}

\begin{thm}\label{T:1.4}
Let $X$ be a Banach space with a shrinking Schauder frame
$(x_i,f_i)_{i=1}^\infty$.
  Let $Z$ be a Banach space with basis $(z_i)_{i=1}^\infty$ such that $(x_i,f_i)_{i=1}^\infty$ is strongly shrinking relative to $(z_i)_{i=1}^\infty$. Let $(p_i)_{i=1}^\infty\in[\N]^\omega$ and
$(\delta_i)_{i=1}^\infty\subset (0,1)$ with $\delta_i\searrow0$.
Then there exists
$(q_i)_{i=1}^\infty,(N_i)_{i=1}^\infty\in[\N]^\omega$ such that for
any $(k_i)_{i=0}^\infty\in[\N]^\omega$ and $y^*\in S_{X^*}$, there
exists $y^*_i\in X^*$ and $t_i\in (N_{k_{i-1}-1},N_{k_{i-1}})$ for
all $i\in\N$ with $N_0=0$ and $t_0=0$ so that the following hold
\begin{enumerate}
\item[(a)] $y^*=\sum_{i=1}^\infty y^*_i$\\
and for all $\ell\in\N$ we have
\item[(b)]either $\|y^*_\ell\|<\delta_\ell$ or $\sup_{p_{q_{t_{\ell-1}}}\geq n\geq m}\|\sum_{j=m}^n y^*_\ell(x_j)f_j\|<\delta_\ell\|y^*_\ell\|$ and\\
$\sup_{n\geq m\geq p_{q_{t_{\ell}}}}\|\sum_{j=m}^n
y^*_\ell(x_j)f_j\|<\delta_\ell\|y_\ell^*\|$,
\item[(c)] $\|P_{[p_{q_{N_{k_\ell}}}, p_{q_{N_{k_{\ell+1}}}})}\circ S^*(y^*_{\ell-1}+y^*_\ell+y^*_{\ell+1}-y^*)\|_{Z^*}<\delta_\ell ,$
\end{enumerate}
where $P_{I}$ is the projection operator $P_I:[z^*_i]\rightarrow
[z^*_i]$ given by $P_I(\sum a_i z^*_i)=\sum_{i\in I}a_i z^*_i$ for
all $\sum a_i z^*_i\in[z^*_i]$ and all intervals $I\subseteq\N$.

\end{thm}

\begin{proof}
By Theorems \ref{T:CL2.3} and \ref{T:CL1.8},
$(f_i,x_i)_{i=1}^\infty$ is a boundedly complete frame for $X^*$
with associated basis $(z_i^*)_{i=1}^\infty$. After renorming, we
may assume without loss of generality that $(z_i)_{i=1}^\infty$ is
bimonotone. We let $K$ be the frame constant of
$(f_i,x_i)_{i=1}^\infty$. Let $\vp_i\searrow0$ such that
$2\vp_{i+1}<\vp_{i}<\delta_i$ and $(1+K)\vp_{i}<\delta^2_{i+1}$ for
all $i\in\N$.

By repeatedly applying Lemma \ref{L:1} to the frame
$(x_i,f_i)_{i=1}^\infty$ of $X$,
 we may choose $(q_k)_{k=1}^\infty\in[\N]^\omega$ such that for all $k\in\N$,

\begin{equation}\label{E:1}
\sup_{n\geq m\geq p_{q_{k+1}}> p_{q_{k}}\geq n_0\geq m_0}\|\sum_{i=m_0}^{n_0} f_i(\sum_{j=m}^n f_j(x) x_j)x_i\|\leq\vp_k\|x\|\qquad\textrm{ for all }x\in X.
\end{equation}
By Lemma \ref{L:2.2}, after possibly passing to a subsequence of
$(q_k)_{k=1}^\infty$, we may assume that for all $k\in\N$,
\begin{equation}\label{E:3}
\sup_{n\geq m\geq p_{q_{k+1}}}\|\sum_{i=1}^{p_{q_k}}(\sum_{j=m}^n x^*(x_j)f_j(x_i))z_i^*\|\leq\vp_k\|x^*\|\qquad\textrm{ for all }x^*\in X^*.
\end{equation}

By applying Lemma \ref{L:0} to the frame $(x_i,f_i)_{i=1}^\infty$ of
$X$,  after possibly passing to a subsequence of
$(q_k)_{k=1}^\infty$, we may assume that for all $k\in\N$,
\begin{equation}\label{E:2}
\sup_{n\geq m\geq p_{q_{k+1}}> p_{q_{k}}\geq n_0\geq m_0}\|\sum_{i=m}^n f_i(\sum_{j=m_0}^{n_0}f_j(x)x_j)x_i\|\leq\vp_{k+1}\|x\|\qquad\textrm{ for all }x\in X.
\end{equation}

By Lemma \ref{L:2.1}, after possibly passing to a subsequence of $(q_k)_{k=1}^\infty$, we may assume that
for all $k\in\N$,
\begin{equation}\label{E:4}
\sup_{p_{q_{k}}> n\geq m\geq 1}\|\sum_{i=p_{q_{k+1}}}^{\infty}(\sum_{j=m}^n x^*(x_j)f_j(x_i))z^*_i\|\leq\vp_{k+1}\|x^*\|\qquad\textrm{ for all }x^*\in X^*.
\end{equation}

By Lemma \ref{L:1.3}, there exists $(N_i)_{i=0}^\infty\in[\N]^\omega$
 such that $N_0=0$ and for all $x^{*}\in X^{*}$ and for all $k\in\N$ there exists $t_k\in\N$ such that $N_k< t_k <N_{k+1}$ and
 $\sup_{p_{q_{t_k-1}}\leq n\leq m <p_{q_{t_k+1}}}\|\sum_{i=n}^m x^{*}(x_i) f_i\|<\vp_k\|x^*\|$.

Let $(k_i)_{i=0}^\infty\in[\N]^\omega$ and $y^*\in S_{X^*}$.  For each $i\in\N$, we choose $t_i\in (N_{k_{i}},N_{k_{i+1}})$
 with $t_0=1$
such that
\begin{equation}\label{E:5}
\sup_{ p_{q_{t_i+1}}> n\geq m \geq p_{q_{t_i-1}}}\|\sum_{j=m}^n
y^{*}(x_j) f_j\|<\vp_i.
\end{equation}
We now set $y^*_i=\sum_{j=p_{q_{t_{i-1}}}}^{p_{q_{t_i}}-1}y^*(x_j)f_j$ for all $i\in\N$.
We have that,
$$\sum_{i=1}^\infty y^*_i=\sum_{i=1}^\infty\sum_{j=p_{q_{t_{i-1}}}}^{p_{q_{t_i}}-1}y^*(x_j)f_j=\sum_{j=1}^\infty y^*(x_j)f_j=y^*.$$
Thus $(a)$ is satisfied.  In order to prove $(b)$, we let $\ell\in\N$ and assume that $\|y^*_\ell\|>\delta_\ell$.
Let $m,n\in\N$ such that $n\geq m\geq p_{q_{t_{\ell}}}$.
To prove property (b), we consider the following inequalities.
\begin{align*}
\|\sum_{j=m}^n y^*_\ell(x_j)f_j\|&= \|\sum_{j=m}^n \sum_{i=p_{q_{t_{\ell-1}}}}^{p_{q_{t_{\ell}}}-1}y^*(x_i)f_i(x_j)f_j\|\\
&\leq \|\sum_{j=m}^n \sum_{i=p_{q_{t_\ell-1}}}^{p_{q_{t_{\ell}}}-1}y^*(x_i)f_i(x_j)f_j\|+\|\sum_{j=m}^n \sum_{i=p_{q_{t_{\ell-1}}}}^{p_{q_{t_{\ell}-1}}-1}y^*(x_i)f_i(x_j)f_j\|\\
&\leq K\|\sum_{i=p_{q_{t_\ell-1}}}^{p_{q_{t_{\ell}}}-1}y^*(x_i)f_i\|+\|\sum_{j=m}^n \sum_{i=p_{q_{t_{\ell-1}}}}^{p_{q_{t_{\ell}-1}}-1}y^*(x_i)f_i(x_j)f_j\|\\
&< K \vp_{t_\ell}+\vp_{t_{\ell}} \qquad\qquad\qquad\textrm{ by }(\ref{E:5})\textrm{ and }(\ref{E:2})\\
&<
(1+K)\vp_{t_\ell}\|y^*_\ell\|/\delta_\ell<(1+K)\vp_{\ell}\|y^*_\ell\|/\delta_\ell<\delta_\ell\|y^*_\ell\|.
\end{align*}
Thus $\sup_{n\geq m\geq p_{q_{t_\ell}}}\|\sum_{j=m}^n
y^*_\ell(x_j)f_j\|<\delta_\ell\|y^*_\ell\|$, proving one of the
inequalities in (b). We now assume that $\ell>1$, and let $m,n\in\N$
such that $p_{q_{t_{\ell-1}}}\geq n\geq m$. To prove the remaining
inequality in (b), we consider the following.
\begin{align*}
\|\sum_{j=m}^n y^*_\ell(x_j)f_j\|&= \|\sum_{j=m}^n \sum_{i=p_{q_{t_{\ell-1}}}}^{p_{q_{t_{\ell}}}-1}y^*(x_i)f_i(x_j)f_j\|\\
&\leq \|\sum_{j=m}^n \sum_{i=p_{q_{t_{\ell-1}+1}}}^{p_{q_{t_{\ell}}}-1}y^*(x_i)f_i(x_j)f_j\|+\|\sum_{j=m}^n \sum_{i=p_{q_{t_{\ell-1}}}}^{p_{q_{t_{\ell-1}+1}}-1}y^*(x_i)f_i(x_j)f_j\|\\
&\leq \|\sum_{j=m}^n \sum_{i=p_{q_{t_{\ell-1}+1}}}^{p_{q_{t_{\ell}}}-1}y^*(x_i)f_i(x_j)f_j\| + K \|\sum_{i=p_{q_{t_{\ell-1}}}}^{p_{q_{t_{\ell-1}+1}}-1}y^*(x_i)f_i\|\\
&< \vp_{t_{\ell-1}+1}+ K \vp_{\ell-1}\qquad\qquad\qquad\textrm{ by }(\ref{E:1})\textrm{ and }(\ref{E:5})\\
&< (\vp_{t_{\ell-1}+1}+ K
\vp_{\ell-1})\|y^*_\ell\|/\delta_\ell<(1+K)\vp_{\ell-1}\|y^*_\ell\|/\delta_\ell<\delta_\ell\|y^*_\ell\|.
\end{align*}
Thus $\sup_{p_{q_{t_{\ell-1}}}\geq n\geq m} \|\sum_{j=m}^n
y^*_\ell(x_j)f_j\|<\delta_\ell\|y^*_\ell\|$, and hence all of
 $(b)$ is satisfied.
To prove $(c)$, we now consider the following,
\begin{align*}
\|P_{[p_{q_{N_{k_{\ell}}}}, p_{q_{N_{k_{\ell+1}}}})}&S^*(y^*_{\ell-1}+y^*_\ell+y^*_{\ell+1}-y^*)\|_{Z^*}=
\|\sum_{i=p_{q_{N_{k_\ell}}}}^{p_{q_{N_{k_{\ell+1}}}}-1}(y^*_{\ell-1}+y^*_\ell+y^*_{\ell+1}-y^*)(x_i)z^*_i\|\\
&=\|\sum_{i=p_{q_{N_{k_\ell}}}}^{p_{q_{N_{k_{\ell+1}}}}-1}\left(\sum_{j=1}^{p_{q_{t_{\ell-2}}}-1}y^*(x_j)f_j(x_i)+\sum_{j=p_{q_{t_{\ell+1}}}}^{\infty}y^*(x_j)f_j(x_i)\right)z^*_i\|\\
&\leq\|\sum_{i=p_{q_{N_{k_\ell}}}}^{p_{q_{N_{k_{\ell+1}}}}-1}\left(\sum_{j=1}^{p_{q_{t_{\ell-2}}}-1}y^*(x_j)f_j(x_i)\right)z^*_i\|+\|\sum_{i=p_{q_{N_\ell}}}^{p_{q_{N_{k_{\ell+1}}}}-1}\left(\sum_{j=p_{q_{t_{\ell+1}}}}^{\infty}y^*(x_j)f_j(x_i)\right)z^*_i\|\\
&\leq\|\sum_{i=p_{q_{N_{k_\ell}}}}^{\infty}\left(\sum_{j=1}^{p_{q_{t_{\ell-2}}}-1}y^*(x_j)f_j(x_i)\right)z^*_i\|+\|\sum_{i=1}^{p_{q_{N_{k_{\ell+1}}}}}\left(\sum_{j=p_{q_{t_{\ell+1}}}}^{\infty}y^*(x_j)f_j(x_i)\right)z^*_i\|\\
&<\vp_{N_{k_\ell}-1}+\vp_{N_{k_{\ell}+1}}<\vp_\ell\qquad\textrm{ by
}(\ref{E:4})\textrm{ and }(\ref{E:3}).
\end{align*}
Thus $(c)$ is satisfied.
\end{proof}

Properties of coordinate systems for Banach spaces such as frames,
bases and FDDs can impose certain structure on infinite dimensional
subspaces. For our purposes, this structure can be intrinsically
characterized in terms of even trees of vectors \cite{OSZ1}. In
order to index even trees, we define
$$T_\infty^{\mathrm{even}}=\{(n_1, . . . ,n_{2\ell}):n_1<\cdot\cdot\cdot<n_{2\ell} \mbox{ are in } \N \mbox{ and } \ell\in\N\}.$$
If X is a Banach space, an indexed family $(x_\alpha)_{\alpha\in
T_\infty^{\mathrm{even}}}\subset X$ is called an \emph{even tree}.
Sequences of the form
$(x_{(n_1,...,n_{2\ell-1},k)})^\infty_{k=n_{2\ell-1}+1}$ are called
\emph{nodes}. This should not be confused with the more standard
terminology where a node would refer to an individual member of the
tree. Sequences of the form
$(n_{2\ell-1},x_{(n_1,...,n_{2\ell})})^\infty_{\ell=1}$ are called
\emph{branches}. A \emph{normalized tree}, i.e. one with
$\|x_\alpha\|=1$ for all $\alpha\in T_\infty^{\mathrm{even}}$, is
called \emph{weakly null} (or \emph{$w^*$-null}) if every node is a
weakly null (or $w^*$-null) sequence.

Given $1>\vp>0$ and $A\subset (\N\times S_{X^*})^\omega$, we let
$A_\vp=\{(l_i,y^*_i)\in(\N\times S_{X^*})^\omega: \exists
(k_i,x^*_i)\in A\textrm{ such that } k_i\leq \ell_i,
\|x_i^*-y_i^*\|<\vp 2^{-i}\forall i\in\N\},$ and we let
$\overline{A}_\vp$ be the closure of $A_\vp$ in $(\N\times
S_{X^*})^\omega$. We consider the following game between players $S$
(subspace chooser) and $P$ (point chooser). The game has an infinite
sequence of moves; on the $n^{th}$ move $S$ picks $k_n\in\N$ and a
cofinite dimensional $w^*$-closed subspace $Z_n$ of $X^*$ and $P$
responds by picking an element $x^*_n\in S_{X^*}$ such that
$d(x^*_n,Z_n)<\vp 2^{-n}$. S wins the game if the sequence
$(k_i,x_i)_{i=1}^\infty$ the players generate is an element of
$\overline{A}_{5\vp}$, otherwise $P$ is declared the winner. This is
referred to as the $(A,\vp)$-game and was introduced in \cite{OSZ1}.
The following proposition is essentially an extension of Proposition
2.6 in \cite{FOSZ} from FDDs to frames, and relates properties of
$w^*$-null even trees and winning strategies of the $(A,\vp)$-game
to blockings of a frame.

\begin{prop}\label{P:skipping}
Let $X$ be an infinite-dimensional Banach space with a shrinking Schauder frame $(x_i,f_i)_{i=1}^\infty$. Let
$A\subseteq (\N\times S_{X^*})^\omega$.  The following are equivalent.
\begin{enumerate}

\item For all $\vp>0$ there exists $(K_i)_{i=1}^\infty\in[\N]^\omega$ and $\bar{\delta}=(\delta_i)\subset(0,1)$ with $\delta_i\searrow0$
 and $(p_i)_{i=1}^\infty\in[\N]^\omega$ such that if $(y^*_i)_{i=1}^\infty\subset S_{X^*}$ and $(r_i)_{i=0}^\infty\in[\N]^\omega$ so that
 $\sup_{ p_{r_{i-1}+1}\geq n\geq m\geq 1}\|\sum_{j=m}^n y_i^*(x_j) f_j\|<\delta_i$
and
 $\sup_{n\geq m\geq p_{r_{i}}}\|\sum_{j=m}^n y_i^*(x_j) f_j\|<\delta_i$ for all $i\in\N$
 then $(K_{r_{i-1}},y^*_i)\in\overline{A}_\vp$.
\item For all $\vp>0$, $S$ has a winning strategy for the $(A,\vp)$-game.
\item For all $\vp>0$ every normalized $w^*$-null even tree in $X^*$ has a branch in $\overline{A}_\vp$.
\end{enumerate}
\end{prop}

\begin{proof}
The equivalences $(2)\Longleftrightarrow(3)$ are given in \cite{FOSZ}.

We now assume $(1)$ holds, and will prove $(3)$.  Let $\vp>0$ and
let $(x^*_{(n_1,...,n_{2\ell})} )_{(n_1,...,n_{2\ell})\in
T^{even}_\infty}$ be a $w^*$-null even tree in $X^*$.  There exists
$(K_i)_{i=1}^\infty\in[\N]^\omega$ and
$\bar{\delta}=(\delta_i)\subset(0,1)$ with $\delta_i\searrow0$
 and $(p_i)_{i=1}^\infty\in[\N]^\omega$ such that if $(y^*_i)_{i=1}^\infty\subset S_{X^*}$ and $(r_i)_{i=0}^\infty\in[\N]^\omega$ so that
 $\sup_{p_{r_{i-1}+1}\geq n\geq m\geq1}\|\sum_{j=m}^n y_i^*(x_j) f_j\|<\delta_i$
and
 $\sup_{n\geq m\geq p_{r_{i}}}\|\sum_{j=m}^n y_i^*(x_j) f_j\|<\delta_i$ for all $i\in\N$
 then $(K_{r_{i-1}},y^*_i)\in\overline{A}_\vp$.

We shall construct by induction sequences
$(r_i)_{i=0}^\infty,(n_i)_{i=1}^\infty\in[\N]^\omega$ such that
$K_{r_i}=n_{2i+1}$  and
 $\sup_{p_{r_{i-1}+1}\geq n\geq m\geq 1}\|\sum_{j=m}^n x_{(n_1,...,n_{2i})}^*(x_j) f_j\|<\delta_i$
and
 $\sup_{n\geq m\geq p_{r_{i}}}\|\sum_{j=m}^n y_i^*(x_j) f_j\|<\delta_i$ for all $i\in\N$.  To start, we let $r_0=1$ and $n_1=K_{1}$.  Now, if $\ell\in\N$ and $(r_i)_{i=0}^\ell$ and
$(n_i)_{i=1}^{2\ell+1}$ have been chosen, then using that
$(x^*_{(n_1,...,n_{2\ell+1},k)})_{k=n_{2\ell+1}+1}^\infty$ is
$w^*$-null, we may choose $n_{2\ell+2}>n_{2\ell+1}$ such that
$\|x^*_{(n_1,...,n_{2\ell+1},n_{2\ell+2})}(x_j)f_j\|<(p_{r_{\ell}}+1)^{-1}\delta_{\ell+1}$.
Thus, $\sup_{p_{r_{\ell}+1}\geq n\geq m\geq1}\|\sum_{j=m}^n
x^*_{(n_1,...,n_{2\ell+1},n_{2\ell+2})}(x_j) f_j\|<\delta_{\ell+1}$.
As $(x_j,f_j)_{i=1}^\infty$ is a Schauder frame, we may choose
$r_{\ell+1}>r_\ell$ such that
$$\sup_{n\geq m\geq p_{r_{\ell+1}}}\|\sum_{j=m}^n
x^*_{(n_1,...,n_{2\ell+1},n_{2\ell+2})}(x_j)
f_j\|<\delta_{\ell+1}.$$  We then let $n_{2\ell+2}=K_{r_{\ell+1}}$.
Thus, our sequences $(r_i)_{i=0}^\infty$ and $(n_i)_{i=1}^\infty$
may be constructed by induction to satisfy the desired properties,
giving us that
$(n_{2i-1},x^*_{(n_1,...,n_{2i})})_{i=1}^\infty=(K_{r_{i-1}},x^*_{(n_1,...,n_{2i})})_{i=1}^\infty\in\overline{A}_\vp$.

We now assume $(2)$ holds, and will prove $(1)$.  Let $\vp>0$ and
assume that player $S$ has a winning strategy for the
$(A,\vp)$-game.  That is, there exists an indexed collection
$(k_{(x^*_1,...,x^*_\ell)})_{(x^*_1,...,x^*_\ell)\in X^{<\N}}$ of
natural numbers, and an indexed collection
$(X^*_{(x^*_1,...,x^*_\ell)})_{(x^*_1,...,x^*_\ell)\in X^{<\N}}$ of
co-finite dimensional $w^*$-closed subsets of $X^*$ such that if
$(x^*_i)_{i=1}^\infty\subset S_{X^*}$ and
$d(x_i^*,X^*_{(x^*_1,...,x^*_i)})<\frac{1}{10}\vp2^{-i}$ for all
$i\in\N$ then
$(k_{(x^*_1,...,x^*_i)},X^*_{(x^*_1,...,x^*_i)})_{i=1}^\infty\in\overline{A}_{\vp/2}$
and $(k_{(x^*_1,...,x^*_i)})_{i=1}^\infty\in[\N]^\omega$.

We  construct by induction $(K_i)_{i=1}^\infty\in[\N]^\omega$,
$(p_i)_{i=1}^\infty\in[\N]^\omega$,
$(\delta_i)_{i=1}^\infty\in(0,1)^\omega$ and a nested collection
$(D_i)_{i=1}^\infty\subset [X^{<\omega}]^{\omega}$ such that $D_i$
is $\frac{1}{20}\vp2^{-i}$-dense in $[f_j]_{j=1}^{p_i}$ and if
$(y^*_i)_{i=1}^\infty\subset S_{X^*}$ and
$(r_i)_{i=0}^\infty\in[\N]^\omega$ so that
 $\sup_{p_{r_{i-1}+1}\geq n\geq m\geq 1}\|\sum_{j=m}^n y_i^*(x_j)
 f_j\|<\delta_i$ and $\sup_{n\geq m\geq p_{r_{i}}}\|\sum_{j=m}^n y_i^*(x_j) f_j\|<\delta_i$ for all $i\in\N$,
 and $x^*_i\in D_{r_{i+1}}$ such that $\|y^*_i-x^*_i\|<\frac{1}{20}\vp2^{-i}$ for all $i\in\N$, then $K_{r_{i-1}}\geq k_{(x^*_1,...,x^*_{i-1})}$, and $d(x_i^*,X^*_{(x^*_1,...,x^*_i)})<\frac{1}{10}\vp2^{-i}$.
This would give that
$(k_{(x^*_1,...,x^*_i)},X^*_{(x^*_1,...,x^*_i)})_{i=1}^\infty\in\overline{A}_{\vp/2}$.
Hence, $(K_{r_{i-1}},y^*_i)_{i=1}^\infty\in\overline{A}_{\vp}$ as
$K_{r_{i-1}}\geq k_{(x^*_1,...,x^*_{i-1})}$ and $\|y^*_i-
x^*_i\|<\frac{1}{20}\vp2^{-i}$ for all $i\in\N$.  Thus all that
remains is to show that $(K_i)_{i=1}^\infty\in[\N]^\omega$,
$(p_i)_{i=1}^\infty\in[\N]^\omega$, and $(D_i)_{i=1}^\infty\subset
[X^{<\omega}]^{\omega}$ may be constructed inductively with the
desired properties.

We start by choosing $K_1=k_\emptyset$. As $(x_i,f_i)_{i=1}^\infty$
is a shrinking Schauder frame for $X$ and $X^*_{\emptyset}\subset
X^*$ is co-finite dimensional and $w^*$-closed, by Lemma \ref{L:1}
there exists $p_1\in\N$ and $\delta_1>0$ such that if $\sup_{
p_{1}\geq n\geq m\geq 1}\|\sum_{j=m}^n y^*(x_j) f_j\|<\delta_1$ for
some $y^*\in S_{X^*}$ then $d(y^*,X^*_{\emptyset})<\frac{1}{20}\vp$.
We then let $D_1$ be some finite $\frac{1}{20}\vp$-net in
$[f_i]_{i=1}^{p_1}$. Now we assume $n\in\N$ and that
$(K_i)_{i=1}^n\in[\N]^{<\omega}$, $(p_i)_{i=1}^n\in[\N]^{<\omega}$,
$(\delta_i)_{i=1}^n\in(0,1)^{<\omega}$ and $(D_i)_{i=1}^n\subset
[X^{<\omega}]^{<\omega}$ have been suitably chosen. As
$(x_i,f_i)_{i=1}^\infty$ is a shrinking Schauder frame for $X$ and
$X^*_{(x^*_1,...,x^*_\ell)}\subset X^*$ is co-finite dimensional and
$w^*$-closed for all $(x^*_1,...,x^*_\ell)\in [D_n]^{<\omega}$, by
Lemma \ref{L:1} there exists $p_{n+1}\in\N$ and $\delta_{n+1}>0$
such that if $\sup_{p_{n+1}\geq n\geq m\geq 1}\|\sum_{j=m}^n
y^*(x_j) f_j\|<\delta_{n+1}$ for some $y^*\in S_{X^*}$ then
$d(y^*,\cap_{(x^*_1,...,x^*_\ell)\in [D_n]^{<\omega}}
X^*_{(x^*_1,...,x^*_\ell)})<\frac{1}{20}\vp2^{-n-1}$. We then let
$K_{n+1}=\max_{(x^*_1,...,x^*_\ell)\in
[D_n]^{<\omega}}k_{(x^*_1,...,x^*_\ell)}$ and let $D_{n+1}$ be a
finite $\frac{1}{20}\vp2^{-n-1}$-net in $[f_i]_{i=1}^{p_{n+1}}$.
\end{proof}

\section{Upper and lower estimates}\label{S:3}

Let $X$ be a Banach space, $V = (v_i)_{i=1}^\infty$ be a normalized
$1$-unconditional basis, and $1 \le C < \infty$. We say that $X$
satisfies \emph{subsequential $C$-$V$-upper tree estimates} if every
weakly null even tree $(x_\alpha)_{\alpha\in
T_\infty^{\mathrm{even}}}$ in $X$ has a branch
$(n_{2\ell-1},x_{(n_1,...,n_{2\ell})})^\infty_{\ell=1}$ such that
$(x_{(n_1,...,n_{2\ell})})^\infty_{\ell=1}$ is $C$-dominated by
$(v_{n_{2\ell-1}})_{\ell=1}^\infty.$ We say that X satisfies
\emph{subsequential $V$-upper tree estimates} if it satisfies
subsequential $C$-$V$-upper tree estimates for some $1 \le C <
\infty$. If X is a subspace of a dual space, we say that X satisfies
\emph{subsequential $C$-$V$-lower $w^*$ tree estimates} if every
$w^*$-null even tree $(x_\alpha)_{\alpha\in
T_\infty^{\mathrm{even}}}$ in $X$ has a branch
$(n_{2\ell-1},x_{(n_1,...,n_{2\ell})})^\infty_{\ell=1}$ such that
$(x_{(n_1,...,n_{2\ell})})^\infty_{\ell=1}$ $C$-dominates
$(v_{n_{2\ell-1}})_{\ell=1}^\infty.$

A basic sequence $V = (v_i)_{i=1}^\infty$ is called $C$-{\em right
dominant} if for all sequences $m_1 < m_2 < \cdot \cdot \cdot$ and
$n_1 < n_2 < \cdot \cdot \cdot$ of positive integers with $m_i \le
n_i$ for all $i \in \N$ the sequence $(v_{m_i})_{i=1}^\infty$ is
$C$-dominated by $(v_{n_i})_{i=1}^\infty$. We say that
$(v_i)_{i=1}^\infty$ is {\em right dominant} if for some $C \ge 1$
it is $C$-right dominant.

\begin{lem}\cite[Lemma 2.7]{FOSZ}\label{L:FOSZ2.7}
Let $X$ be a Banach space with separable dual, and let $V = (v_i)_{i=1}^\infty$ be a normalized,
1-unconditional, right dominant basic sequence. If $X$ satisfies subsequential $V$-upper tree estimates,
then $X^*$ satisfies subsequential $V^*$-lower $w^*$ tree estimates.
\end{lem}

Let $Z$ be a Banach space with an FDD $(E_i)_{i=1}^\infty$, let
$V=(v_i)_{i=1}^\infty$ be a normalized $1$-unconditional basis, and
let $1\le C<\infty$. We say that $(E_i)_{i=1}^\infty$ satisfies
\emph{subsequential $C$-$V$-upper block estimates} if every
normalized block sequence $(z_i)_{i=1}^\infty$ of
$(E_i)_{i=1}^\infty$ in $Z$ is $C$-dominated by
$(v_{m_i})_{i=1}^\infty$, where $m_i=\min\supp_E(z_i)$ for all
$i\in\N.$ We say that $(E_i)_{i=1}^\infty$ satisfies
\emph{subsequential $C$-$V$-lower block estimates} if every
normalized block sequence $(z_i)_{i=1}^\infty$ of
$(E_i)_{i=1}^\infty$ in $Z$ $C$-dominates $(v_{m_i})_{i=1}^\infty$,
where $m_i=\min\supp_E(z_i)$ for all $i\in\N.$ We say that
$(E_i)_{i=1}^\infty$ satisfies \emph{subsequential $V$-upper (or
lower) block estimates} if it satisfies subsequential $C$-$V$-upper
(or lower) block estimates for some $1\le C<\infty$.

Note that if $(E_i)_{i=1}^\infty$ satisfies subsequential
$C$-$V$-upper block estimates and $(z_i)_{i=1}^\infty$ is a
normalized block sequence with $\max \supp_E(z_{i-1}) < k_i \le \min
\supp_E(z_i)$ for all $i > 1$, then $(z_i)_{i=1}^\infty$ is
$C$-dominated by $(v_{k_i})_{i=1}^\infty$ (and a similar remark
holds for lower estimates).

Subsequential $V^{(*)}$-upper block estimates and subsequential
$V$-lower block estimates are dual properties, as shown in the
following proposition from \cite{OSZ1}.

\begin{prop}\cite[Proposition 2.14]{OSZ1}\label{P:OSZ2.14}
Assume that $Z$ has an FDD $(E_i)_{i=1}^\infty$, and let
$V=(v_i)_{i=1}^\infty$ be a normalized and 1-unconditional basic
sequence. The following statements are equivalent:
\begin{enumerate}
\item[(a)] $(E_i)_{i=1}^\infty$ satisfies subsequential $V$-lower block estimates in $Z$.
\item[(b)] $(E^*_i)_{i=1}^\infty$ satisfies subsequential $V^{(*)}$-upper block estimates in $Z^{(*)}$.
\end{enumerate}
(Here subsequential $V^{(*)}$-upper estimates are with respect to
$(v^*_i)_{i=1}^\infty$, the sequence of biorthogonal functionals to
$(v_i)_{i=1}^\infty$).

Moreover, if $(E_i)_{i=1}^\infty$ is bimonotone in $Z$, then the
equivalence holds true if one replaces, for some $C \geq 1$,
$V$-lower estimates by $C$-$V$-lower estimates in (a) and
$V^{(*)}$-upper estimates by $C$-$V^{(*)}$-upper estimates in (b).
\end{prop}
Note that by duality, Proposition \ref{P:OSZ2.14} holds if we
interchange the words ``upper" and ``lower".

Let $(x_i,f_i)_{i=1}^\infty$ be a shrinking Schauder frame for a
Banach space $X$, and let $(v_i)$ be a normalized, 1-unconditional,
block stable, 1-right dominant, and shrinking basic sequence.  For
any $C>0$, we may apply Proposition \ref{P:skipping} to the set
$A=\{(n_i,x^*_i)_{i=1}^\infty\in(\N\times S_{X^*})^\omega:\,
(x^*_i)_{i=1}^\infty\, C-\textrm{dominates
}(v^*_{n_i})_{i=1}^\infty\}$ to obtain the following corollary.

\begin{cor}\label{C:blocking}
Let $(x_i,f_i)_{i=1}^\infty$ be a shrinking Schauder frame for a
Banach space $X$, and let $V=(v_i)_{i=1}^\infty$ be a normalized,
1-unconditional, block stable, right dominant, and shrinking basic
sequence.  The following are equivalent.
\begin{enumerate}
\item  There
exists $C>0$, $(K_i)_{i=1}^\infty,(p_i)_{i=1}^\infty\in[\N]^\omega$,
and $\bar{\delta}=(\delta_i)\subset(0,1)$ with $\delta_i\searrow0$
 such that if $(y^*_i)_{i=1}^\infty\subset S_{X^*}$ and $(r_i)_{i=0}^\infty\in[\N]^\omega$ so that
 $\sup_{1\leq n\leq m\leq p_{r_{i-1}+1}}\|\sum_{j=n}^m y_i^*(x_j) f_j\|<\delta_i$
and
 $\sup_{p_{r_{i}}\leq n\leq m}\|\sum_{j=n}^m y_i^*(x_j) f_j\|<\delta_i$
 then $(y^*_i) \succeq_C (v^*_{K_{r_{i-1}}})$.\\

\item $X$ satisfies subsequential $V$ upper tree estimates.
\end{enumerate}

\end{cor}

Let $Z$ be a Banach space with a basis $(z_i)_{i=1}^\infty$, let
$(p_i)_{i=1}^\infty\in[\N]^\omega$, and let $V = (v_i)_{i=1}^\infty$
be a normalized 1-unconditional basic sequence. The space $Z_V(p_i)$
is defined to be the completion of $\coo$ with respect to the
following norm $\|\cdot\|_{Z_V} :$

 $$\|\sum a_i z_i\|_{Z_V}=\max_{M\in\N, 1\leq r_0\leq r_1<\cdots<{r_M}}
 \left\Vert\sum_{i=1}^M\left\|\sum_{j=p_{r_i}}^{p_{{r_{i+1}}-1}}a_j z_j\right\|_{Z}v_{{r_i}}\right\Vert_{V}\qquad \textrm{ for }(a_i)\in\coo.
$$

The following proposition from \cite{OSZ1} is what makes the space
$Z_V$ essential for us. Recall that in \cite{OSZ1} a basic sequence,
$(v_i)_{i=1}^\infty$, is called {\em $C$-block stable} for some $C
\geq 1$ if any two normalized block bases $(x_i)_{i=1}^\infty$ and
$(y_i)_{i=1}^\infty$ with $\max(supp(x_i) , supp(y_i)) <
min(supp(x_{i+1}) , supp(y_{i+1}))$ for all $i \in\N$ are
$C$-equivalent. We say that $(v_i)_{i=1}^\infty$ is {\em block
stable} if it is $C$-block stable for some constant $C$.  We will
make use of the fact that the property of block stability dualizes.
That is, if $(v_i)_{i=1}^\infty$ is a block stable basic sequence
 then $(v^*_i)_{i=1}^\infty$ is also a block stable basic sequence.
Another simple, though important, consequence of a normalized basic
sequence
 $(v_i)_{i=1}^\infty$ being block stable,
is that there exists a constant $c\geq 1$ such that
 $(v_{n_i})_{i=1}^\infty$ is $c$-equivalent to $(v_{n_{i+1}})_{i=1}^\infty$
 for all $(n_i)_{i=1}^\infty\in[\N]^\omega$.
Block stability has been considered before in various forms
and under different names. In particular, it has been called the blocking
principle \cite{CJT} and the shift property \cite{CK} (see \cite{FR} for alternative forms).
The following proposition recalls some properties of $Z_V(p_i)$ which were
shown in \cite{OSZ1}.

\begin{prop}\cite[Corollary 3.2, Lemmas 3.3, 3.5, and 3.6]{OSZ1}\label{L:OSZ3.2}
Let $V =(v_i)_{i=1}^\infty$ be a normalized, 1-unconditional, and
C-block stable basic sequence. If $Z$ is a Banach space with a basis
$(z_i)_{i=1}^\infty$ and $(p_i)_{i=1}^\infty\in[\N]^\omega$, then
$(z_i)_{i=1}^\infty$ satisfies $2C$-$V$-lower block estimates in
$Z_V(p_i)$. If the basis $(v_i)_{i=1}^\infty$ is boundedly complete
then $(z_i)_{i=1}^\infty$ is a boundedly complete basis for $Z_V
(p_i)$. If the basis $(v_i)_{i=1}^\infty$ is shrinking and
$(z_i)_{i=1}^\infty$ is shrinking in $Z$, then $(z_i)_{i=1}^\infty$
is a shrinking basis for $Z_V (p_i)$.

If $U=(u_i)_{i=1}^\infty$  is a normalized, 1-unconditional and
block-stable basic sequence such that $(v_i)_{i=1}^\infty$ is
dominated by $(u_i)_{i=1}^\infty$ and $(z_i)_{i=1}^\infty$ satisfies
subsequential $U$-upper block estimates in $Z$, then
$(z_i)_{i=1}^\infty$ also satisfies subsequential $U$-upper block
estimates in $Z_V(p_i)$.
\end{prop}

\begin{thm}\label{T:1.5}
Let $X$ be a Banach space with a shrinking Schauder frame $(x_i,f_i)$ which is strongly shrinking relative to some Banach space
$Z$ with basis $(z_i)$ and bounded operators $T:X\rightarrow Z$ and $S:Z\rightarrow X$
 defined by $T(x)=\sum f_i(x)z_i$ for all $x\in X$ and $S(z)=\sum z^*_i(z) x_i$ for all $z\in Z$. Let $V=(v_i)_{i=1}^\infty$ be a normalized, 1-unconditional, block stable, right dominant, and shrinking basic sequence.
If $X$ satisfies subsequential $V$ upper tree estimates, then there
exists $(n_i)_{i=1}^\infty,(K_i)_{i=1}^\infty\in[\N]^\omega$ such
that $Z^*_{(v^*_{K_i})}(n_i)$ is an associated space of
$(f_i,x_i)_{i=1}^\infty$ with bounded operators $S^*:X^*\rightarrow
Z^*_{(v^*_{K_i})}(n_i)$ and
 $T^*:Z^*_{(v^*_{K_i})}(n_i)\rightarrow X$
 given by $S^*(x^*)=\sum x^*(x_i)z^*_i$ for all $x^*\in X^*$ and $S^*(z^*)=\sum z^*(z_i) f_i$ for all $z^*\in Z^*_{(v^*_{K_i})}(n_i)$.
\end{thm}

\begin{proof}
After renorming, we may assume that
 the basis $(z_i)_{i=1}^\infty$ is bimonotone.
 The sequence $(f_i,x_i)_{i=1}^\infty$ is a boundedly complete Schauder frame for $X^*$ by Theorem \ref{T:CL2.3}, and we have that $X^*$ satisfies subsequential
 $V^*$ lower $w^*$ tree estimates by Lemma \ref{L:FOSZ2.7}.
 By Theorem \ref{T:CL1.8}, the basis $(z^*_i)_{i=1}^\infty$ is an associated basis for $(f_i,x_i)_{i=1}^\infty$ with bounded operators $S^*:X^*\rightarrow Z^*$ and
 $T^*:Z^*\rightarrow X$
 given by $S^*(x^*)=\sum x^*(x_i)z^*_i$ for all $x^*\in X^*$ and $S^*(z^*)=\sum z^*(z_i) f_i$ for all $z^*\in Z^*$.
 Let $\vp>0$.
By Corollary \ref{C:blocking}, there exists $C>0$,
$(K_i)_{i=1}^\infty,(p_i)_{i=1}^\infty\in[\N]^\omega$, and
$\bar{\delta}=(\delta_i)\subset(0,1)$ with $\delta_i\searrow0$ and
$\sum \delta_i<\vp$
 such that if $(y^*_i)_{i=1}^\infty\subset S_{X^*}$ and $(r_i)_{i=0}^\infty\in[\N]^\omega$ so that
 $\sup_{1\leq n\leq m\leq p_{r_{i-1}+1}}\|\sum_{j=n}^m y_i^*(x_j) f_j\|<\delta_i$
and
 $\sup_{p_{r_{i}}\leq n\leq m}\|\sum_{j=n}^m y_i^*(x_j) f_j\|<\delta_i$
 then $(y^*_i) \succeq_C (v^*_{K_{r_{i-1}}})$.  We apply Theorem \ref{T:1.4} to $(x_i,f_i)_{i=1}^\infty$, $(p_i)_{i=1}^\infty\in[\N]^\omega$, and
 $(\delta_i)_{i=1}^\infty\subset(0,1)$ to obtain $(q_i),(N_i)_{i=1}^\infty\in[\N]^\omega$ satisfying the conclusion of Theorem \ref{T:1.4}.

 By Theorem \ref{T:CL1.8}, $(z^*_i)_{i=1}^\infty$ is an associated basis for $(f_i,x_i)_{i=1}^\infty$, and we denote the norm on
 $Z^*$ by $\|\cdot\|_{Z^*}$.  We block the basis $(z^*_i)_{i=1}^\infty$ into an FDD by setting $E_i=span_{j\in[p_{q_{N_i}},p_{q_{N_{i+1}}})}z^*_j$ for all $i\in\N$.
 We now define a new norm $\|\cdot\|_{\bar{Z}^*}$  on $span(z^*_i)_{i=1}^\infty$ by
 $$\|\sum a_i z^*_i\|_{\bar{Z}^*}=\max_{M\in\N, 1\leq r_0\leq r_1<\cdots<{r_M}}
 \left\Vert\sum_{i=1}^M\left\|\sum_{j=p_{q_{N_{r_i}}}}^{p_{q_{N_{r_{i+1}}}}-1}a_j z^*_j\right\|_{Z^*}v^*_{K_{N_{r_i}}}\right\Vert_{V^*}\qquad\textrm{ for all }(a_i)\in\coo.
$$

 We let $\bar{Z}^*$ be the completion of
$span(z^*_i)_{i=1}^\infty$ under the norm $\|\cdot\|_{\bar{Z}^*}$.
Note that $\|z^*\|_{Z^*}\leq\|z^*\|_{\bar{Z}^*}$ for all $z^*\in
Z^*$.  As $(v_i^*)_{i=1}^\infty$ is block stable, there exists a
constant $c\geq 1$ such that $(v^*_{n_i})_{i=1}^\infty\approx_c
(v^*_{n_{i+1}})_{i=1}^\infty$ for all
$(n_i)_{i=1}^\infty\in[\N]^\omega$. We now show that
$\|S^*(y^*)\|_{\bar{Z}^*}\leq{(1+2\vp)3cC} \|y^*\|$ for all $y^*\in
X^*$.

Let $y^*\in X^*$,
$M\in\N$, and  $1\leq r_0\leq r_1<\cdots<{r_M}$.  We will show that $(1+2\vp)3cC\|S\|\|S(y^*)\|\geq
\left\Vert\sum_{i=1}^M\left\|\sum_{j=p_{q_{N_{r_i}}}}^{p_{q_{N_{r_{i+1}}}}-1} y^*(x_j)z^*_j\right\|v^*_{p_{N_{r_i}}}\right\Vert_{V^*}$.
By Theorem \ref{T:1.4}, there exists $y^*_i\in X^*$ and $t_i\in (N_{r_{i-1}-1},N_{r_{i-1}})$ for all
$i\in\N$ with $N_0=0$ and $t_0=0$ such that
\begin{enumerate}
\item[(a)] $y^*=\sum_{i=1}^\infty y^*_i$\\
and for all $i\in\N$ we have
\item[(b)]either $\|y^*_i\|<\delta_i$ or $\sup_{p_{q_{t_{i-1}}}\geq n\geq m}\|\sum_{j=m}^n y^*_i(x_j)f_j\|<\delta_i\|y^*_i\|$ and\\
$\|\sup_{n\geq m\geq p_{q_{t_{i}}}}\|\sum_{j=m}^n y^*_i(x_j)f_j\|<\delta_i\|y_i^*\|$,
\item[(c)] $\|P^*_{[p_{q_{N_{k_i}}}, p_{q_{N_{k_{i+1}}}})}\circ S^*(y^*_{i-1}+y^*_i+y^*_{i+1}-y^*)\|_{Z^*}<\delta_i $
\end{enumerate}
We let $A=\{i\in\N\,:\, \|y^*_i\|>\delta_i\}$.  By our choice of
$(p_i)_{i=1}^\infty$, we have that $(y^*_i/\|y^*_i\|)_{i\in
A}\succeq_C (v^*_{K_{r_{i-1}}})_{i\in A}$.  Thus, $C\|\sum_{i\in A}
y^*_i\|\geq \|\sum _{i\in A}\|y^*_i\|v^*_{K_{r_{i-1}}}\|_{V^*}$.  We
now obtain the following lower estimate for $\|y^*\|$.
\begin{align*}
\|y^*\|&=\|\sum y^*_i\|  \qquad\qquad\textrm{ by }(a)\\
&\geq \|\sum_{i\in A} y^*_i\|+\sum_{i\not\in A}\|y^*_i\|-\vp\qquad\qquad\textrm{ as }\sum\delta_i<\vp\\
&\geq \frac{1}{C}\|\sum_{i\in
A}\|y^*_i\|v^*_{K_{r_{i-1}}}\|_{V^*}+\sum_{i\not\in
A}\|y^*_i\|-\vp\qquad\qquad\textrm{ as } C\|\sum_{i\in A}
y^*_i\|\geq \|\sum _{i\in A}\|y^*_i\|v^*_{K_{r_{i-1}}}\|_{V^*}\\
&\geq \frac{1}{C}\|\sum \|y^*_i\|v^*_{K_{r_{i-1}}}\|_{V^*}-\vp\\
&\geq \frac{1}{3cC}(\|\sum \|y^*_i\|v^*_{K_{r_{i-1}}}\|_{V^*}+\|\sum \|y^*_{i-1}\|v^*_{K_{r_{i-1}}}\|_{V^*}
+\|\sum \|y^*_{i+1}\|v^*_{K_{r_{i-1}}}\|_{V^*})-\vp\\
&\geq \frac{1}{3cC}\|\sum \|y^*_{i-1}+y^*_i+y^*_{i+1}\|v^*_{K_{r_{i-1}}}\|_{V^*}-\vp\\
&\geq \frac{1}{3cC\|S\|}\|\sum \|S^*(y^*_{i-1}+y^*_i+y^*_{i+1})\|_{Z^*}v^*_{K_{r_{i-1}}}\|_{V^*}-\vp\qquad\qquad\textrm{ as $(v^*_i)$ is 1-unconditional}\\
&\geq \frac{1}{3cC\|S\|}\|\sum \|P^*_{[p_{q_{N_i}}, p_{q_{N_{i+1}}})}S^*(y^*_{i-1}+y^*_i+y^*_{i+1})\|_{Z^*}v^*_{K_{r_{i-1}}}\|_{V^*}-\vp\qquad\textrm{ as $(z_i)$ is bimonotone}\\
&\geq \frac{1}{3cC\|S\|}\|\sum \|\sum_{j=p_{q_{N_{i}}}}^{p_{q_{N_{i+1}}}-1} y^*(x_j)z_j\|v^*_{K_{r_{i-1}}}\|_{V^*}-2\vp\qquad\qquad\textrm{ by }(c)\\
\end{align*}
Thus we have that $\|S^*y^*\|_{\bar{Z}^*}\leq {(1+\vp)3cC\|S\|} \|y^*\|$ for all $y^*\in X^*$.  Hence, $S^*:X^*\rightarrow {\bar{Z}^*}$ is
an isomorphism.
We have that $T^*:Z^*\rightarrow X^*$ is bounded, and hence $T^*:\bar{Z}^{*}\rightarrow X^*$ is bounded as well, as
$\|z^*\|_{Z^*}\leq\|z^*\|_{\bar{Z}^{*}}$ for all $z^*\in Z^*$.  Thus, $\bar{Z}^{*}$ is an associated space of $X^*$.
\end{proof}

The following theorem can be thought of as an extension of Theorem
1.1 in \cite{FOSZ} to frames.

\begin{thm}\label{T:1.6}
Let $X$ be a Banach space with a shrinking Schauder frame
$(x_i,f_i)_{i=1}^\infty$. Let $(v_i)_{i=1}^\infty$ be a normalized,
1-unconditional, block stable, right dominant, and shrinking basic
sequence. If $X$ satisfies subsequential $(v_i)_{i=1}^\infty$ upper
tree estimates, then there exists
$(n_i)_{i=1}^\infty,(K_i)_{i=1}^\infty\in[\N]^\omega$ and an
associated space $Z$ with a shrinking basis $(z_i)_{i=1}^\infty$
such that the FDD $(span_{j\in[n_i,n_{i+1})}z_i)_{i=1}^\infty$
satisfies subsequential
 $(v_{K_i})_{i=1}^\infty$ upper block estimates.
\end{thm}

\begin{proof}
As $(x_i,f_i)_{i=1}^\infty$ is a shrinking Schauder frame, it is strongly shrinking relative to some associated basis
$(z_i)_{i=1}^\infty$ for a Banach space $Z$ by Theorem \ref{T:1.2}.  We thus have bounded operators $T:X\rightarrow Z$ and $S:Z\rightarrow X$
 defined by $T(x)=\sum f_i(x)z_i$ for all $x\in X$ and $S(z)=\sum z^*_i(z) x_i$ for all $z\in Z$.  By Theorem \ref{T:1.5},
 $Z^*_{(v^*_{K_i})}(n_i)$ is an associated space of $(f_i,x_i)_{i=1}^\infty$ with bounded operators $S^*:X^*\rightarrow Z^*_{(v^*_{K_i})}(n_i)$ and
 $T^*:Z^*_{(v^*_{K_i})}(n_i)\rightarrow X$
 given by $S^*(x^*)=\sum x^*(x_i)z^*_i$ for all $x^*\in X^*$ and $S^*(z^*)=\sum z^*(z_i) f_i$ for all $z^*\in Z^*_{(v^*_{K_i})}(n_i)$.
   We define $\bar{Z}$ as the
 completion of $[z_i]_{i=1}^\infty$ under the norm  $\|\sum a_i z_i\|_{\bar{Z}}=\sup_{z^*\in B_{Z^*_{(v^*_{K_i})}(n_i)}}z^*(\sum a_i z_i)$.  As $(z^*_i)_{i=1}^\infty$ is a boundedly
  complete basis of $Z^*_{(v^*_{K_i})}(n_i)$ by Lemma \ref{L:OSZ3.2}, we have that $(z_i)_{i=1}^\infty$ is a shrinking basis for $\bar{Z}$ and that the dual of $\bar{Z}$
  is $Z^*_{(v^*_{K_i})}(n_i)$.  We now prove that $\bar{Z}$ is an associated space of $(x_i,f_i)_{i=1}^\infty$.

  If $(x^*_i)_{i=1}^\infty\subset X^*$ and $x^*_i\rightarrow_{w^*}0$ then $(S^*(x_i^*))_{i=1}^\infty$ converges $w^*$ to 0 as a sequence
in $Z^*$.  Thus $((S^*(x_i^*))(z_j))_{i=1}^\infty$ converges to 0
for all $j\in\N$.  Lemma \ref{L:OSZ3.2} gives that
$(z^*_i)_{i=1}^\infty$
 is a boundedly complete basis for $Z^*_{(v^*_{K_i})}(n_i)$, and hence converging $w^*$ to 0 in $Z^*_{(v^*_{K_i})}(n_i)$ is equivalent to converging coordinate wise to 0.  Hence, $(S^*(x_i^*))_{i=1}^\infty$
converges $w^*$ to 0 in $Z^*_{(v^*_{K_i})}(n_i)$.  Thus
$S^*:X^*\rightarrow Z^*_{(v^*_{K_i})}(n_i)$ is $w^*$ to $w^*$
continuous, and hence is a dual operator.  Thus, $S:
\bar{Z}\rightarrow X$.

If $(z^*_i)_{i=1}^\infty\subset Z^*_{(v^*_{K_i})}(n_i)$ converges
$w^*$ to 0 in $Z^*_{(v^*_{K_i})}(n_i)$, then $(z^*_i)_{i=1}^\infty$
converges coordinate wise to 0, and hence $(z^*_i)_{i=1}^\infty$
converges $w^*$ to 0 in $Z^*$.  Thus,
$(T^*(z^*_i))_{i=1}^\infty\rightarrow_{w^*}0$ in $X^*$.  We thus
have that $T^*:Z^*_{(v^*_{K_i})}(n_i)\rightarrow X^*$ is $w^*$ to
$w^*$ continuous, and is hence a dual operator.  Thus,
$T:X\rightarrow \bar{Z}$ is bounded.  This gives us that, $\bar{Z}$
is an associated space for $(x_i,f_i)_{i=1}^\infty$. By Lemma
\ref{L:OSZ3.2} we have that the FDD
$(span_{j\in[p_{N_i},p_{N_{i+1}})}z^*_j)_{i=1}^\infty$ satisfies
$(v^*_{K_i})$ lower block estimates in $\overline{Z}^*$, and hence
$(span_{j\in[p_{N_i},p_{N_{i+1}})}z_j)_{i=1}^\infty$ satisfies
$(v_{K_i})_{i=1}^\infty$ upper block estimates in $\overline{Z}$.
\end{proof}

The following theorem can be thought of as an extension of Theorem
4.6 in \cite{OSZ1} to frames.

\begin{thm}\label{T:1.7}
Let $X$ be a Banach space with a shrinking and boundedly complete
Schauder frame $(x_i,f_i)_{i=1}^\infty$. Let $(u_i)_{i=1}^\infty$ be
a normalized, 1-unconditional, block stable, right dominant, and
shrinking basic sequence, and let
 $(v_i)_{i=1}^\infty$ be a normalized, 1-unconditional, block stable, left dominant, and shrinking basic sequence such that $(u_i)$ dominates $(v_i)$.
Then $X$ satisfies subsequential $(u_i,v_i)_{i=1}^\infty$ tree
estimates, if and only if there exists
$(n_i)_{i=1}^\infty,(K_i)_{i=1}^\infty\in[\N]^\omega$ and a
reflexive associated space $Z$ with a basis $(z_i)_{i=1}^\infty$
such that the FDD $(span_{j\in[n_i,n_{i+1})}z_j)_{i=1}^\infty$
satisfies subsequential
 $(u_{K_i},v_{K_i})_{i=1}^\infty$ block estimates.
\end{thm}

\begin{proof}

By Theorem \ref{T:1.6}, $(x_i,f_i)_{i=1}^\infty$ has an associated
basis $(z_i)_{i=1}^\infty$ such that there exists
$(m_i)_{i=1}^\infty,(k_i)_{i=1}^\infty\in[\N]^\omega$ and an
associated space $Z$ with a shrinking basis $(z_i)_{i=1}^\infty$
such that the FDD $(span_{j\in[n_i,n_{i+1})}z_j)_{i=1}^\infty$
satisfies subsequential
 $(u_{K_i})_{i=1}^\infty$ upper block estimates.  We have that $(f_i,x_i)_{i=1}^\infty$ is a shrinking frame for $X^*$ which is strongly shrinking relative to the
 associated basis $(z^*_i)_{i=1}^\infty$ by Corollary \ref{C:1.1}.  The space $X$ satisfying subsequential $(v_i)_{i=1}^\infty$ lower tree estimates implies
 that $X^*$ satisfies subsequential $(v^*_i)_{i=1}^\infty$ upper tree estimates.  Thus we may apply Theorem \ref{T:1.5} to
the space $X^*$, the frame $(f_i,x_i)_{i=1}^\infty$ and the
associated basis $(z^*_i)_{i=1}^\infty$ to
 obtain
$(n_i)_{i=1}^\infty,(K_i)_{i=1}^\infty\in[\N]^\omega$ such that
$Z_{(v_{K_i})}(n_i)$ is an associated space of
$(x_i,f_i)_{i=1}^\infty$.  Furthermore, we may assume that
$(n_i)_{i=1}^\infty$ is a subsequence of $(m_i)_{i=1}^\infty$ and
that $(K_i)_{i=1}^\infty$ is a subsequence of $(k_i)_{i=1}^\infty$
as $(v_i)_{i=1}^\infty$ is left dominant. By Lemma \ref{L:OSZ3.2},
the FDD $(span_{j\in[n_i,n_{i+1})}z_j)_{i=1}^\infty$ satisfies
subsequential
 $(u_{K_i},v_{K_i})_{i=1}^\infty$ block estimates.
\end{proof}

We now show that Theorem \ref{T:mainUpper} follows immediately from
Theorem \ref{T:1.6}.  For the same reason, Theorem
\ref{T:mainReflexive} follows immediately from Theorem \ref{T:1.7}.

\begin{proof}[Proof of Theorem \ref{T:mainUpper}] Let $(x_i, f_i)_{i=1}^\infty$
be a shrinking Schauder frame for a Banach space $X$ and let
$\alpha$ be a countable ordinal. The equivalences
$(a)\Longleftrightarrow(b)$ are given in \cite{OSZ2}.

Let $(t_i)_{i=1}^\infty$ be the unit vector basis for
$T_{\alpha,c}$, where $0<c<1$ is some constant. If a Banach space
has an FDD satisfying subsequential $(t_{K_i})_{i=1}^\infty$ upper
block estimates for some sequence
$(K_i)_{i=1}^\infty\in[\N]^\omega$, then it satisfies subsequential
$T_{\alpha,c}$-upper tree estimates. Thus if $(x_i,
f_i)_{i=1}^\infty$ has an associated space with an FDD satisfying
subsequential $(t_{K_i})_{i=1}^\infty$ upper block estimates, then
$X$ embeds into a Banach space satisfying subsequential
$T_{\alpha,c}$-upper tree estimates. Hence, $X$ itself would satisfy
subsequential $T_{\alpha,c}$-upper tree estimates.  Thus
$(c)\Longrightarrow(b)$.

The unit vector basis, $(t_i)_{i=1}^\infty$, for $T_{\alpha,c}$ is a
normalized, 1-unconditional, block stable, right dominant, and
shrinking basic sequence.  Thus $(b)\Longrightarrow(c)$ by Theorem
\ref{T:1.6}.
\end{proof}

\end{document}